\documentclass[a4paper,11pt]{amsart}

\usepackage{amsmath}
\usepackage{amssymb}
\usepackage{amsthm}
\usepackage{verbatim}
\usepackage[all]{xy}
\usepackage{enumerate}
\usepackage{mathbbol}

\usepackage{color}

\newtheorem{thm}{Theorem}[section]
\newtheorem{prop}[thm]{Proposition}
\newtheorem{cor}[thm]{Corollary}
\newtheorem{lem}[thm]{Lemma}

\theoremstyle{definition}

\newtheorem{dfn}[thm]{Definition}

\newtheorem{rmk}[thm]{Remark}

\numberwithin{equation}{section}
\newcommand{\bG}{\mathbb{G}}
\newcommand{\bK}{\mathbb{K}}

\newcommand{\id}{\textrm{id}}

\newcommand{\bH}{\mathsf{H}}

\newcommand{\bT}{\mathsf{T}}

\newcommand{\tr}{\textup{tr}}

\newcommand{\N}{\mathbb{N}}

\newcommand{\bc}{\mathbb{C}}
\newcommand{\bz}{\mathbb{Z}}
\newcommand{\br}{\mathbb{R}}
\newcommand{\Irr}{{\rm Irr}}
\newcommand{\Pol}{{\rm Pol}}

\newcommand{\ltwo}{\ell^2(\GGamma)}
\newcommand{\lone}{\ell^1(\GGamma)}
\newcommand{\bt}{\mathbb{T}}

\newcommand{\Linf}{\ell^\infty(\GGamma)}

\newcommand{\ot}{\otimes}

\newcommand{\supp}{{\rm supp}}

\newcommand{\cA}{\mathcal{A}}

\newcommand{\GGamma}{\mathbb{\Gamma}}
\newcommand{\bGGamma}{\widehat{\GGamma}}

\DeclareMathOperator{\C}{C}

\title[Quantum $\C^\ast$-uniqueness]{On $\C^*$-completions of discrete quantum group rings}

\author{Martijn Caspers}
\address{TU Delft, EWI/DIAM, 
	P.O.Box 5031, 
	2600 GA Delft, 
	The Netherlands}
\email{m.p.t.caspers@tudelft.nl}

\author{Adam Skalski}
\address{Institute of Mathematics of the Polish Academy of Sciences,
	ul.~\'Sniadeckich 8, 00--656 Warszawa, Poland}
\email{a.skalski@impan.pl}


\begin{document}

\begin{abstract}
	We discuss just infiniteness of $\C^*$-algebras associated to discrete quantum groups and relate it to the $\C^*$-uniqueness of the quantum groups in question, i.e.\ to the uniqueness of a $\C^*$-completion of the underlying Hopf $^*$-algebra. It is shown that duals of $q$-deformations of simply connected semisimple compact Lie groups  are never $\C^*$-unique. On the other hand we present an example of a discrete quantum group which is not locally finite and yet is $\C^*$-unique.
\end{abstract}

\maketitle

%compact simple quantum group is sufficient?

\section{Introduction}
 
The following definition, modelled on the notion of just infiniteness for groups, was introduced by R.\,Grigorchuk, M.\,Musat and M.\,R\o rdam in the recent paper \cite{GrigorchukMusatRordam}:
a $\C^*$-algebra $A$ is called \emph{just infinite} if all its proper quotients are finite-dimensional. Just infinite $\C^*$-algebras were classified in \cite{GrigorchukMusatRordam} in terms of their ideal spaces. For group $\C^*$-algebras it was shown that this notion is closely related to the existence of `nontrivial' $\C^*$-completions of group rings (see Corollary \ref{corfulljustinfinite} below). In particular it was observed in \cite{GrigorchukMusatRordam} that the only obvious obstruction for the existence of different   $\C^*$-completions of a group ring is local finiteness of the underlying group. This naturally raises a question whether for any discrete group $\Gamma$ which is not locally finite one can construct  different   $\C^*$-completions of $\bc[\Gamma]$; if the latter holds we say for short that $\Gamma$ is not $\C^*$-unique. The preprint \cite{AlexeevKyed} takes some steps towards verifying this conjecture, producing different completions of group rings for various classes of discrete groups, such as infinite groups of polynomial growth or groups with a central element of infinite order.
 
Our note considers analogous questions in the setup of discrete \emph{quantum} groups. If $\GGamma$ is a discrete quantum group, one can still consider the associated quantum group ring $\bc[\GGamma]$ and its reduced and universal $\C^*$-completions $\C^*_r(\GGamma)$ and $\C^*(\GGamma)$. Again one can ask about just infiniteness of these $\C^*$-algebras; as in the classical case it is not difficult to see that these are connected first to amenability of $\GGamma$, and then to the uniqueness of $\C^*$-completions of $\bc[\GGamma]$. This naturally raises the question which discrete quantum groups are $\C^*$-unique and whether the equivalence of $\C^*$-uniqueness and local finiteness has a chance to be true in the quantum realm (the backward implication obviously holds). We show by an explicit construction  that the duals of so-called $q$-deformations of classical compact Lie groups, such as Woronowicz's $SU_q(2)$, are never $\C^*$-unique. On the other hand we answer the above question in the negative, by showing that a certain crossed product construction, combining $SU_q(2)$ and the noncommutative torus, yields a discrete quantum group which is not locally finite and yet is $\C^*$-unique.

The plan of the paper is as follows: in Section \ref{Sect=QG} we recall some quantum group terminology, establish equivalence of amenability of a discrete quantum group $\GGamma$ with $\C^*_r(\GGamma)$ admitting a finite dimensional representation, discuss some basic facts related to just infiniteness of quantum group operator algebras and note that locally finite discrete quantum groups are $\C^*$-unique. In Section \ref{Sect=qDef} we show that the duals of $q$-deformations are never $\C^*$-unique. Finally in Section \ref{Sect=Example} we produce an example of a $\C^*$-unique discrete quantum group which is not locally finite.

All $^*$-algebras we study will be unital, and by a representation of a $^*$-algebra $A$ we will understand a unital $^*$-homomorphism $\pi:A \to B(\bH)$, where $\bH$ is a Hilbert space. We write $\N_0$ for $\N \cup\{0\}$. Scalar products will be linear on the right. When we talk about a $\C^*$-norm on a $^*$-algebra strictly speaking we mean a pre-$\C^*$-norm.

\section{Discrete quantum groups and just infiniteness of their group $\C^*$-algebras}\label{Sect=QG}

Throughout the paper $\bG$ will denote a compact quantum group in the sense of Woronowicz, and $\GGamma:=\widehat{\bG}$ will be the discrete quantum group dual to $\bG$. For precise definitions and all the related terminology we refer (for example) to Section 2 of \cite{DKSS}; we will follow the conventions of that paper. We will be mainly interested in the quantum group ring $\bc[\GGamma]$ (in other words the Hopf $^*$-algebra $\Pol(\bG)$), and its reduced and universal completions $\C^*_r(\GGamma)$ and $\C^*(\GGamma)$ (in other words $\C(\bG)$ and $\C^u(\bG)$), with the latter being the universal enveloping $\C^*$-algebra of $\bc[\GGamma]$. The reduced algebra $\C^*_r(\GGamma)$ acts on the Hilbert space $\ell^2(\GGamma)$, as does the `algebra of bounded functions on $\GGamma$', the von Neumann algebra $\ell^\infty(\GGamma)$. The predual of the latter is denoted $\lone$.

%, whose coproduct we will denote by $\Delta$. Recall that $\GGamma$ is said to be amenable if there exists a state $m:\ell^\infty(\GGamma)\to \bc$ such that for every $\omega \in \ell^1(\GGamma)= \ell^\infty(\GGamma)_*$ we have
%\[ m ((\omega \otimes \id)( \Delta(x))) = m (( \id\otimes \omega)( \Delta(x)))\omega(x), \;\;\; x \in \ell^\infty(\GGamma). \]

 Recall that
\[ \ell^\infty (\GGamma)= \prod_{\alpha\in\Irr_{\bG}} M_{n_{\alpha}},\]
where $\Irr_{\bG}$ denotes the set of equivalence classes of irreducible unitary representations of $\bG$. For each $\alpha \in \Irr_{\bG}$ we choose a representative, i.e.\ a unitary matrix $U^{\alpha} = (u^{\alpha}_{i,j})_{i,j=1}^{n_{\alpha}} \in M_{n_{\alpha}}(\bc[\GGamma])$. The matrix units in $M_{n_{\alpha}}\subset \ell^\infty(\GGamma)$ will be denoted by $e_{i,j}^{\alpha}$.

The multiplicative unitary of $\GGamma$ is the unitary $W\in B(\ltwo \ot \ltwo)$ given by the formula:
\[ W = \sum_{\alpha \in \Irr_{\bG}} \sum_{i,j=1}^{n_{\alpha}}e_{j,i}^{\alpha} \ot (u^{\alpha}_{i,j})^*\]
%(this formula can be easily deduced from the more familiar expression for the multiplicative unitary $W^{\bG}$ of the dual quantum group and the relation $W=\sigma(W^{\bG})^*$).

The coproduct of $\GGamma$, a coassociative  normal unital $^*$-homomorphism $\Delta:\Linf \to \Linf \overline{\ot} \Linf$, is implemented by $W$ via the following formula:
\begin{equation}
\Delta(x) = W^* (1 \ot x) W, \;\;\; x \in \Linf.
\end{equation}
Given a functional $\phi \in \lone$ we define the (normal, bounded) maps $L_\phi: \Linf \to \Linf$ and $R_\phi:\Linf \to \Linf$  via the formulas
\[ L_\phi= (\phi \ot \id)\circ \Delta, \;\;\; R_\phi= (\id \ot \phi)\circ \Delta.\]

A discrete quantum group $\GGamma$ is called \emph{amenable} if it admits a bi-invariant mean, i.e.\ a state $m\in \Linf^*$, such that for all $\phi \in \lone$
there is
\[m \circ L_{\phi} = m \circ R_{\phi} = \phi(1) m.\]
%\end{deft}

By \cite{Vaesetal} a discrete quantum group $\GGamma$ is amenable if it admits a left invariant mean $m\in \Linf^*$: a state such that for each $\phi \in \lone$
there is $m \circ L_{\phi} =  \phi(1) m$. In fact it suffices to check the last formula for the functionals of the form $\widehat{e_{i,j}^{\alpha}}$, $\alpha \in \Irr_{\bG}$, $i,j=1,\ldots,n_{\alpha}$, as the latter are linearly dense in $\lone$, and the map $\phi \mapsto L_\phi$ is a (complete) isometry. Thus we will need the following explicit form of the map $L_{\phi}$ for $\phi=\widehat{e_{i,j}^{\alpha}}$:
\begin{equation} \label{Lphi}
L_{\phi} (x) =
% (\phi \ot \id)\circ \Delta (x) =  (\phi \ot \id) \left( W^* ( 1 \ot e_{k,l}^{\beta}) W \right) =
%\\&= (\phi \ot \id) \left( \left( \sum_{\gamma \in \Irr_{\bG}} \sum_{m,p=1}^{n_{\gamma}}e_{m,p}^{\gamma} \ot u^{\gamma}_{m,p} \right) (1 \ot x )
%\left( \sum_{\kappa \in \Irr_{\bG}} \sum_{r,s=1}^{n_{\kappa}}e_{s,r}^{\kappa} \ot (u^{\kappa}_{r,s})^* \right)  \right)
%\\&= (\phi \ot \id)  \left( \sum_{\gamma \in \Irr_{\bG}} \sum_{m,p,r=1}^{n_{\gamma}} \left(  e_{m,r}^{\gamma} \ot \left( u^{\gamma}_{m,p}  x (u^{\gamma}_{r,p})^* \right)  %\right) \right)
\sum_{p=1}^{n_{\alpha}} u^{\alpha}_{i,p}  x (u^{\alpha}_{j,p})^*, \;\;\;\;\;\; x \in l^{\infty}(\GGamma).
\end{equation}

A pre-$\C^\ast$-norm on $\mathbb{C}[\GGamma]$ is a norm such that the completion of $\mathbb{C}[\GGamma]$ is a $\C^\ast$-algebra.  

\begin{dfn}
	We call a discrete quantum group $\GGamma$ $\C^\ast$-unique if $\mathbb{C}[\GGamma]$ (i.e. $\Pol(\widehat{\GGamma})$) has a unique pre-$\C^\ast$-norm. We call $\GGamma$ $\C^\ast_r$-unique if there is no pre-$\C^\ast$-norm on  $\mathbb{C}[\GGamma]$  that is properly majorized by the norm of $\C^\ast_r(\GGamma)$. 
\end{dfn}

We have the canonical quotient map $\Lambda:\C^*(\GGamma)\to \C^*_r(\GGamma)$. The following is a combination of results in \cite{BMT} and \cite{Tomatsu}.

\begin{thm} \label{amenableknown}
Let $\GGamma$ be a discrete quantum group. The following conditions are equivalent:
\begin{itemize}
	\item[(i)] $\GGamma$ is amenable;
	\item[(ii)] the quotient map $\Lambda:\C^*(\GGamma)\to \C^*_r(\GGamma)$ is an isomorphism;
	\item[(iii)] the algebra $\C^*_r(\GGamma)$ admits a character.
\end{itemize}	
\end{thm}

%\subsection*{Discrete/compact quantum groups and a new characterisation of amenability}

 We will need the following lemma, due to Biswarup Das, Matt Daws and Pekka Salmi. % the proof is based on the idea of the quantum Bohr compactification from \cite{Soltan}.

\begin{lem}[Das, Daws and Salmi]\label{DDSLemma}
Suppose that $\GGamma$ is a discrete quantum group, $n \in \N$ and $\pi: \bc[\GGamma] \to M_n$ is a representation. Then $\pi$ is invariant under the scaling group: $\pi = \pi \circ \tau_t$, $t \in \br$.
\end{lem}
\begin{proof}
Let $\GGamma$ and $\pi: \C[\GGamma] \to M_n$ be as above. Naturally the representation $\pi$ extends to a representation of $\C^*(\GGamma)$, which we denote by the same symbol. Consider the Kac quotient of $\widehat{\GGamma}$, $\widehat{\GGamma}_\textup{KAC}$, as defined in the appendix of \cite{Soltan} (with the idea attributed to Vaes), and studied later for example in \cite{Daws}, and denote its dual by $\mathbb{\Lambda}$. There is a canonical unital $^*$-homomorphism $\rho:\C^*(\GGamma) \to \C^*(\mathbb{\Lambda})$, intertwining the respective comultiplications. Moreover, as discussed in the Appendix of \cite{Soltan} or in \cite[Proposition 6.5]{Daws}, the representation $\pi$ factorises via $\rho$, so that there is a representation $\pi': \C^*(\mathbb{\Lambda}) \to M_n$ such that $\pi = \pi' \circ \rho$. Now as the scaling group of a Kac type compact quantum group is trivial, and by \cite[Remark 12.1]{Kustermans} the quantum group morphism intertwines the respective (universal) scaling groups, we have
$\rho = \rho \circ \tau_t$, $t \in \br$. The result follows. 
\end{proof}

The above lemma has a simple corollary, using the properties of the unitary and the usual antipode of $\bc[\Gamma]$.

\begin{cor} \label{twistformula}
Suppose that $\GGamma$ is a discrete quantum group, $n \in \N$, $\pi: \bc[\GGamma] \to M_n$ is a representation, and $U=(u_{ij})_{i,j=1}^k\in M_k(\bc[\GGamma])$ is a finite-dimensional unitary representation of $\widehat{\GGamma}$. Then we have the following equality ($i,j=1, \ldots,k$):
\[\pi\left(\sum_{p=1}^k u_{j,p}^* u_{i,p} \right) = \delta_{ij}1_{M_n}.\] 
\end{cor}
\begin{proof}
By Lemma \ref{DDSLemma} the representation $\pi$ is invariant under the scaling group action. This implies that if we restrict $\pi$ to $\bc[\GGamma]$ then we have $\pi \circ R = \pi \circ S$, where $R$ and $S$ denote respectively the unitary and the usual antipode of $\bc[\GGamma]$. We then use the fact that $R^2=\id$ and both antipodes commute to calculate (say for $i,j,p = 1, \ldots,k$)
\[\pi(u_{i,p} ) = \pi \circ R^2(u_{i,p} ) = \pi \circ S \circ R(u_{i,p}) = \pi \circ R \circ S (u_{i,p}) = \pi \circ R(u_{p,i}^*) = \pi \circ S (u_{p,i}^*), \] so that 
\[ 	
\pi(u_{j,p}^* u_{i,p} ) = \pi (u_{j,p}^*) \pi (u_{i,p}) = \pi\circ S(u_{p,j}) \pi \circ S(u_{p,i}^*) =\pi (S (u_{p,j}) S(u_{p,i}^*)) = 
\pi (S(u_{p,i}^* u_{p,j}))\]
and finally
\[ \pi(\sum_{p=1}^k u_{j,p}^* u_{i,p} )	= \pi \circ S(\sum_{p=1}^k u_{p,i}^* u_{p,j} ) = \pi\circ S (\delta_{ij} 1_{\bc[\GGamma]}) = \delta_{ij} 1_{M_n},\]
where we used the unitarity of $U$ and the fact that the antipode is unital.	
\end{proof}

The following result is related to $\C^*_r(\GGamma)$ being just infinite; it extends Theorem 2.8 of \cite{BMT}, i.e.\ the implication (iii)$\Longrightarrow$(i) of Theorem \ref{amenableknown}. Note that the proof of that theorem in \cite{BMT} does not extend to the matrix case considered below; we use rather the idea of an Arveson extension, as in Section 2.6 of \cite{bo}. 
 
 \begin{prop} \label{amenabilityfdreps}
 	A discrete quantum group $\GGamma$ is amenable if and only if $\C^*_r(\GGamma)$ admits a finite-dimensional representation.
\end{prop}
 \begin{proof}
 	The forward implication is well-known (and follows for example from Theorem 2.8 of \cite{BMT} mentioned above).
 	
 	Suppose then that $\pi: \C^*_r(\GGamma)\to M_n$ is a representation. By Lemma \ref{DDSLemma} we know that $\pi$ is invariant under the scaling group. % We need in fact to consider the `twisted' version of $\pi$, putting $\wt{\pi} = ^T \circ \pi \circ R$, where $^T$ denotes a transposition operation on $M_n$ (with respect to some basis) and $R: \C^*_r(\GGamma) \to \C^*_r(\GGamma)$ is the unitary antipode. Note that $\wt{\pi}$ is also a representation. 
 	We view $\C^*_r(\GGamma)$ as a subalgebra of $B(\ell^2(\GGamma))$ and consider a unital completely positive extension of $\pi$ to the latter algebra, denoted by $\Phi$, so that we have $\Phi: B(\ell^2(\GGamma)) \to M_n$; note that by a standard multiplicative domain argument we know that $\Phi$ is a $\C^*_r(\GGamma)$-bimodule map in the obvious sense. We claim that the state $m: \ell^\infty (\GGamma)\to \bc$ defined by the formula $m = \tr \circ \Phi|_{ \ell^\infty (\GGamma)} $ is the desired left invariant mean.
 	
Fix then a functional $\phi \in \ell^1(\GGamma)$ of the form $\widehat{e_{i,j}^{\alpha}}$, $\alpha \in \Irr_{\bG}$, $i,j=1,\ldots,n_{\alpha}$	  and $x \in \ell^\infty(\GGamma)$ and compute, using the formula \eqref{Lphi}:
\begin{align*} m(L_\phi(x)) &= m \left(\sum_{p=1}^{n_{\alpha}} u^{\alpha}_{i,p}  x (u^{\alpha}_{j,p})^*\right) = 
\sum_{p=1}^{n_{\alpha}} (\tr \circ \Phi) \left( u^{\alpha}_{i,p}  x (u^{\alpha}_{j,p})^*\right) =
\sum_{p=1}^{n_{\alpha}} \tr  \left( \Phi(u^{\alpha}_{i,p}) \Phi(x) \Phi((u^{\alpha}_{j,p})^*) \right)  \\&=
\sum_{p=1}^{n_{\alpha}} \tr  \left( \Phi(x) \Phi((u^{\alpha}_{j,p})^*)  \Phi(u^{\alpha}_{i,p})\right)  =
\sum_{p=1}^{n_{\alpha}} \tr  \left( \Phi(x) \pi((u^{\alpha}_{j,p})^*) \pi(u^{\alpha}_{i,p})\right) \\&=
 \tr  \left( \Phi(x) \sum_{p=1}^{n_{\alpha}} \pi((u^{\alpha}_{j,p})^*u^{\alpha}_{i,p})\right).
 \end{align*} 	
Using now Corollary \ref{twistformula} we get  immediately	
\[  m(L_\phi(x)) = \tr  \left( \Phi(x) \delta_{ij}1_{M_n}\right) = m(x) \delta_{ij} = m(x) \phi(1_{\Linf}).  \] 	
 \end{proof}
  
 %\subsection*{Just-infiniteness for group $\C^*$-algebras of discrete quantum groups} 
  
We are ready to present some basic facts concerning the just infiniteness of group $\C^*$-algebras of discrete quantum groups, essentially following Section 6 of \cite{GrigorchukMusatRordam} and Section 2 of \cite{AlexeevKyed}.
 
 \begin{prop}
 	If $\C^*_r(\GGamma)$ is just infinite then either $\GGamma$ is $\C^*$-simple (i.e.\ $\C^*_r(\GGamma)$ is simple), or $\GGamma$ is amenable.
 \end{prop}
 \begin{proof}
 	Immediate consequence of Proposition \ref{amenabilityfdreps}.
 \end{proof}

 We call a $^*$-algebra $\cA$ $^*$-just infinite if any representation of $\cA$ on a Hilbert space is either injective or has a finite-dimensional image.

 \begin{prop} \label{universal}
 	Let $\cA$ be an infinite-dimensional unital $^*$-algebra admitting a maximal $\C^*$-norm; denote by $A_u$ the corresponding universal $\C^*$-completion. Then $A_u$ is just infinite if and only if $\cA$ is $^*$-just infinite and admits a unique $\C^*$-completion. Furthermore $\cA$ admits a unique $\C^*$-completion if and only if every non-trivial (closed, two-sided) ideal of $A_u$ has a non-trivial intersection with $\cA$.
 \end{prop}
 \begin{proof}
 	The first part follows as in  Proposition 6.3 of \cite{GrigorchukMusatRordam}; we present the proof below for completeness.
 	
 	Assume first that $A_u$ is just infinite. Consider a representation $\pi:\cA\to B(H)$ and its extension to $A_u$. The latter is either injective or has finite-dimensional image. But this means that the original $\pi$ was either injective or had a finite-dimensional image. Hence $\cA$ is $^*$-just infinite. Then note that for  any pre-$\C^*$-norm  on $\cA$ we have  a representation $\pi:
 	A_u \to B(H)$, injective on $\cA$, such that the norm in question equals $\|\cdot\|_\pi$. As the image of $\pi$ is infinite-dimensional, $\pi$ must be injective, so that $\|\cdot\|_\pi = \|\cdot\|_{A_u}$.
 	
 	Suppose then that $\cA$ is $^*$-just infinite and admits a unique $\C^*$-completion. Consider a non-injective representation $\pi: A_u \to B(H)$. Then it is either injective on $\cA$, in which case it would give a different $\C^*$-completion of $\cA$, or it is non-injective on $\cA$, in which case $\pi(\cA)$ is finite-dimensional. But then $\pi(A_u)$ is finite-dimensional and we showed that $A_u$ is just infinite.

 	The last statement follows very easily (see Lemma 2.2 in \cite{AlexeevKyed}).
 \end{proof}

 The first part of the last result implies immediately the following corollary.
 
 \begin{cor} \label{corfulljustinfinite}
 	Let $\GGamma$ be infinite. Then $\C^*(\GGamma)$ is just infinite if and only if $\mathbb{C}[\GGamma]$ is $^*$-just infinite and $\C^*$-unique. Thus if $\C^*(\GGamma)$ is just infinite then $\GGamma$ is amenable.
 \end{cor}

The group $\bz$ is not $\C^*$-unique; in fact, as noted in \cite{GrigorchukMusatRordam} $\bc[\bz]$ does not admit a minimal pre-$\C^*$-norm. This fact (or rather reasons behind it) can be used to prove the following statement (for classical groups shown in Proposition 2.4 of \cite{AlexeevKyed}).

\begin{prop}
 Suppose that $\GGamma$ is a discrete quantum group, containing $\bz$ in its centre (in other words, $\bc[\bz]$ is a Hopf $^*$-subalgebra of the centre of $\bc[\GGamma])$. Then $\GGamma$ is not $\C^*_r$-unique.
\end{prop}
\begin{proof}
Consider the Haar state $h$ on $\bc[\GGamma]$. By the uniqueness its restriction to $\bc[\bz] = \Pol(\bt)$ is given by the integration with respect to the Lebesgue measure on the circle, and we have the inclusions $\C^*_r(\bz)\subset Z( \C^*_r(\GGamma))$ and $\textup{VN}(\bz)\subset  Z(\textup{VN}(\GGamma))$. Arguing as in Proposition 2.4 of \cite{AlexeevKyed} (see also the following section) we obtain a projection $p\in \textup{VN}(\bz)$ such that the restricted representation $\pi: \C^*_r(\GGamma) \to B(p\ell^2(\GGamma))$ is faithful on $\Pol(\bt)$ and is not faithful on $\C^*_r(\bz)$ (so also not on $\C^*_r(\GGamma)$). It remains to note that it is faithful on $\bc[\GGamma]$. This is however easy: consider the $h$-preserving conditional expectation $\mathbb{E}: \textup{VN}(\GGamma) \to \textup{VN}(\bz)$, whose existence follows from the Takesaki's theorem \cite{TakII} and the fact that the latter algebra is obviously contained in the centralizer of $h$. As $\bc[\bz] \subset \bc[\GGamma]$ can be viewed as spanned by certain (one-dimensional) irreducible representations of $\widehat{\GGamma}$, Woronowicz-Peter-Weyl formulae show that  $\mathbb{E}(\bc[\GGamma]) = \bc[\bz]$. Then one can conclude the proof of faithfulness of $\pi$ on $\bc[\GGamma]$ as in Proposition 2.4 of \cite{AlexeevKyed}.

The result now follows from the second statement in Proposition \ref{universal}, if we  consider the kernel of the representation $\pi$.
	
\end{proof}

 In fact the proof above shows that $\GGamma$ is even not $\C^*_r$-unique.  Theorem \ref{amenableknown} shows that for amenable discrete quantum groups $\C^*_r$-uniqueness is the same as $\C^*$-uniqueness, so we will focus on the latter concept.

%We now discuss the class of discrete quantum groups which are $\C^*$-unique for essentially trivial reasons.

 %Question: how much of this can be improved to the algebras associated to quantum  homogenous spaces? In other words I want to understand for example when having a non-unique $\C^*$-norm on $\Pol(\bG/\bK)$ means that $\C(\bG/\bK)$ cannot be just infinite.

 \begin{dfn}\label{localfin}
 	We call a discrete quantum group $\GGamma$ locally finite if  each finite subset $I\subset \subset \Irr(\bGGamma)$ generates a finite fusion ring inside $\Irr(\bGGamma)$.	
 \end{dfn}

 \begin{lem}
 	Consider a   discrete quantum group $\GGamma$. The following are equivalent:
 	\begin{enumerate}
 		\item[(i)] $\GGamma$ is locally finite;  
 		\item[(ii)] every finite subset of $\mathbb{C}[\GGamma]$ is contained in a finite-dimensional sub Hopf*-algebra of $\mathbb{C}[\GGamma]$;
 		\item[(iii)] every finite subset of $\mathbb{C}[\GGamma]$ is contained in a finite-dimensional unital *-subalgebra of $\mathbb{C}[\GGamma]$.
 	\end{enumerate}
 \end{lem}
 \begin{proof}
 	The equivalence of (i) and (ii) is a straightforward consequence of the Woronowicz-Peter-Weyl theory. The implication (ii)$\Longrightarrow$(iii) is trivial. Assume then that (iii) holds. Consider a finite set $F\subset \mathbb{C}[\GGamma]$. By the fundamental theorem on coalgebras there is a finite-dimensional sub-coalgebra (which we may assume to be unital and self-adjoint) $C\subset \mathbb{C}[\GGamma]$ containing $F$. Consider then the algebra generated by $C$. By (iii) it is finite-dimensional (we can choose a finite linear basis in $C$); it is easy to check that it is in fact a $^*$-Hopf subalgebra of $\mathbb{C}[\GGamma]$. 
 \end{proof}
 
Locally finite discrete quantum groups are $\C^*$-unique for essentially trivial reasons.

 \begin{prop}
 	If a discrete quantum group $\GGamma$ is locally finite, then $\mathbb{C}[\GGamma]$ is $\C^*$-unique.
 \end{prop}
 \begin{proof}
 	The easy proof follows exactly as in Proposition 6.7 of \cite{GrigorchukMusatRordam}, using part (iii) of the equivalence in the lemma above. 
 \end{proof}

 In Section 4 we will exhibit an example showing that the converse implication does not hold.

% Discuss the notion/examples of just infinite $C^*$-algebras associated to quantum groups 

 \section{Non-unique completions for $q$-deformations}\label{Sect=qDef}
 In this section we prove that the quantum groups that arise as $q$-deformations of simply connected semisimple compact Lie groups never have a unique $\C^\ast$-closure of their polynomial algebra. We prove this explicitly for $SU_q(2)$ and then treat the general case. Recall that duals of all such $q$-deformations are amenable, as was first shown in \cite{Banica} and then reproved via methods related to Theorem \ref{amenableknown} in the Appendix of \cite{FST}. In particular they are $\C^*$-unique if and only if they are $\C^*_r$-unique.

  Let $\bG_q = SU_q(2), q \in (-1, 1) \backslash \{ 0 \}$. Algebraically $\Pol(\bG_q)$ is  defined as the $\ast$-algebra generated by operators $\alpha$ and $\gamma$ satisfying the relations,
 \begin{equation}\label{Eqn=RelSU2}
 \gamma^\ast \gamma = \gamma \gamma^\ast, \qquad q \gamma^\ast \alpha  = \alpha \gamma^\ast,   \qquad  q \gamma \alpha = \alpha \gamma, \qquad  \alpha^\ast \alpha +  \gamma^\ast \gamma = 1, \qquad \alpha \alpha^\ast +  q^2 \gamma \gamma^\ast = 1,
 \end{equation}
 with comultiplication extending 
 \begin{equation}\label{Eqn=SU2Coproduct}
 \Delta(\alpha) = \alpha \otimes \alpha - q \gamma^\ast \otimes \gamma \qquad {\rm and } \qquad  \Delta(\gamma) = \gamma \otimes \alpha + \alpha^\ast \otimes \gamma.
 \end{equation}  
% We write $\Pol(\bG_q)$ for the $\ast$-algebra of matrix coefficients of $\bG_q$ and 
As before, we write $\C(\bG_q)$ for the $\C^\ast$-closure under the GNS-representation of the Haar state (equivalently, the universal closure, as $\widehat{\bG_q}$ is amenable). It is isomorphic to the $\C^\ast$-algebra generated by the concrete operators acting on $L_2(\{0\} \cup  \bigcup_{k=0}^\infty q^{k}  \mathbb{T})$, with the underlying measure being a sum of the Dirac measure on $0$ and (normalised) Lebesgue measures on each rescaled circle $q^{k}  \mathbb{T}$:
 \[
 ( \widetilde{\alpha} f)(z) = \sqrt{1 - \vert z \vert^2} f( q^{-1} z),  \qquad    (\widetilde{\gamma}f)( z  ) =   z f(z), \qquad z \in \{0\} \cup  \bigcup_{k=0}^\infty q^{k}  \mathbb{T}. 
 \]
 Here we take the convention that $f(q^{-1} z) = 0, z \in \mathbb{T}$. 
 So $\widetilde{\alpha}$ and $\widetilde{\gamma}$ are the images of $\alpha$ and $\gamma$ under the isomorphism in question and the $\ast$-algebra generated by  $\widetilde{\alpha}$ and $\widetilde{\gamma}$ is isomorphic to $\Pol(\bG_q)$. 
 Write  $\mathbb{T}_{\uparrow} = \{ z \in \mathbb{T} \mid \Im(z) \geq 0 \}$ and $\mathbb{T}_{\downarrow} = \mathbb{T} \backslash \mathbb{T}_{\downarrow}$ for the upper half circle and (open) lower half circle respectively.

%*********** Define C$_r^\ast$-unique somewhere else
 
  \begin{thm}
 	The discrete quantum group $\widehat{\bG_q}$ is not $\C^\ast$-unique.  
 \end{thm}
 \begin{proof}
 	Let $A = \C^\ast\langle \widetilde{\gamma} \rangle$ be the (commutative) unital  C$^\ast$-subalgebra of $\C(\bG_q)$ generated by the normal element $\widetilde{\gamma}$. 
 	
 	Assume first that $q>0$. The spectrum of $\widetilde{\gamma}$ is  $\{0\} \cup  \bigcup_{k=0}^\infty q^{k}  \mathbb{T}$, so that we can identify $A$ with $C(\{0\} \cup  \bigcup_{k=0}^\infty q^{k}  \mathbb{T})$. Let then $f \in C(\{0\} \cup  \bigcup_{k=0}^\infty q^{k}  \mathbb{T})$ be a non-zero function supported on $\bigcup_{k=0}^\infty q^{k}  \mathbb{T}_{\uparrow}$, and denote by $x$ the corresponding element of $A$. Let $\pi: A \rightarrow B(L_2(\bigcup_{k=0}^\infty q^{k}  \mathbb{T}_{\downarrow}))$ be the representation of $A$ that is given simply by pointwise multiplication.
 	%; above we identify $L_2(\bigcup_{k=0}^\infty q^{k}  \mathbb{T}_{\downarrow})$ with $\bigoplus_{k=0}^\infty L_2( q^{k}  \mathbb{T}_{\downarrow} )$.   	%Let $x \in A$ be such that the support of  $\pi(x)$ lies in $\bigcup_{k=0}^\infty q^{k}  \mathbb{T}_{\uparrow}$ and let $p \in L_\infty(\bigcup_{k=0}^\infty q^{k}  \mathbb{T})$  be the projection onto $L_2(\bigcup_{k=0}^\infty q^{k}  \mathbb{T}_{\downarrow})$. Note that we can view $A$ as acting on $L_2(\bigcup_{k=0}^\infty q^{k}  \mathbb{T}_{\downarrow}))$ simply 

 	We define a representation $\rho$ of $\Pol(\bG_q)$ on $L_2(\bigcup_{k=0}^\infty q^{k}  \mathbb{T}_{\downarrow})$ as follows. Set 
 	\[
 	\rho(\gamma)(f)(z) = z f(z), \quad \rho(\alpha)(f)(z) = \sqrt{1 - \vert z \vert^2} f(q^{-1} z ), \qquad z \in \bigcup_{k=0}^\infty q^{k}  \mathbb{T}_{\downarrow}, \quad  f \in L_2(\bigcup_{k=0}^\infty q^{k}  \mathbb{T}_{\downarrow}).
 	\] 
 	So $\rho(\gamma)$ and $\rho(\alpha)$ are the restrictions of $\widetilde{\gamma}$ and $\widetilde{\alpha}$ to $L_2(\bigcup_{k=0}^\infty q^{k}  \mathbb{T}_{\downarrow})$. Therefore  these prescriptions for $\rho$ define a representation of $\Pol(\bG_q)$ on $L_2(\bigcup_{k=0}^\infty q^{k}  \mathbb{T}_{\downarrow})$. The representation $\rho$ extends to a representation of $\C(\bG_q)$  and we have $\rho(x) = 0$. 
 	
 	It remains to show that $\rho$ is injective on $\Pol(\bG_q)$. In order to do this, note that by the defining relations of $\bG_q$  a linear basis of $\Pol(\bG_q)$ is given by
 	\[
 	\alpha^k \gamma^m (\gamma^\ast)^n, \qquad k,m,n \in \N_0.
 	\]
 	Take a finite linear combination $y = \sum_{k,m,n} c_{k,m,n} \alpha^k \gamma^m (\gamma^\ast)^n$.  	Suppose that $\rho(y) = 0$. Then the explicit description of $\rho$ shows that for all $k$ we must have $\rho(\sum_{m,n} c_{k,m,n} \alpha^k \gamma^m (\gamma^\ast)^n) = 0$ which is equivalent to the property that for all $k$ we have that   $\rho(\sum_{m,n} c_{k,m,n}  \gamma^m (\gamma^\ast)^n) = 0$. As the functions $z \mapsto z^m \overline{z}^n$ are linearly independent in $C( \bigcup_{k=0}^\infty q^{k}  \mathbb{T}_{\downarrow} )$ we see that this can  happen only if for all $k,m$ and $n$ we have $c_{k,m,n} = 0$.  
 	
 	So $\rho$ is an injective representation of $\Pol(\bG_q)$ and the associated C$^\ast$-norm is properly majorized by the C$^\ast$-norm of $\C(\bG_q)$ as $\rho(x) = 0$ but $x \not = 0$. 
 	
 	As noted by the referee, the construction has to be modified for $q<0$. The modification is however straightforward: we simply replace the set
$ 	\bigcup_{k=0}^\infty q^{k}  \mathbb{T}_{\downarrow}$ by the set    $ \bigcup_{k=0}^\infty \left(q^{2k}  \mathbb{T}_{\downarrow} \cup q^{2k+1}  \mathbb{T}_{\uparrow}\right)$, to take into account the fact that in that case the action of $\alpha$ `flips' the upper/lower sides of the circle. Otherwise the whole argument remains the same.
 \end{proof}

% add precise definition of compact simple quantum groups and the range of $q$. Also check if the $q$ in $\bG_q$ is the same as $\bK_q$ in the next theorem. 
 
Next we generalize this result. Let now $q\in (0,1)$, let $K$ be a simply connected semisimple compact Lie group and let $\bK_q$ denote the $q$-deformation  of $K$ in the sense of \cite{KS}, with  $\Pol(\bK_q)$ the corresponding polynomial algebra. In \cite[Theorem 6.2.7]{KS} the irreducible representations of $\Pol(\bK_q)$ are classified. 
 
 \begin{thm}[Theorem 6.2.7 of \cite{KS}]\label{Thm=MassiveDecomp}
 	We have the following:  	
 \begin{enumerate}	
 	\item \label{Item=One} There is group $W$ generated by a set $S$, called the Weyl group, and a maximal torus $\mathbb{T}^d$ in $\bK_q$ such that for every $w \in W$ with reduced expression $w = s_1 \cdots s_n, s_i \in S,$  and $\tau \in \mathbb{T}^d$ there is a representation 
 \[
\pi_{w,\tau}:= (\pi_{s_1} \otimes \cdots \otimes \pi_{s_n} \otimes \pi_\tau) \circ \Delta_{(n+1)}, 
 \] 
 of $\Pol(\bK_q)$ on the Hilbert space $H_w:=\ell^2(\N_0)^{\otimes n} \otimes \bc$. If $w$ has another reduced expression, then the corresponding representation is unitarily equivalent and we set $\pi_w =  (\pi_{s_1} \otimes \ldots \otimes \pi_{s_n}) \circ \Delta_{(n)}$.  Further all the representations indexed by $W$ and $\mathbb{T}^d$ form a complete set of  mutually inequivalent irreducible representations of $\Pol(\bK_q)$; and $\C(\bK_q)$ is a type $I$ $\C^*$-algebra. 
 \item \label{Item=Two} The representations $\pi_s, s \in S$ factor through the ($s$-dependent) projections $\Pol(\bK_q) \twoheadrightarrow \Pol(\textup{SU}_q(2))$ and representations $\pi_\tau, \tau \in \mathbb{T}^d$ factor through the canonical projection $\Pol(\bK_q) \twoheadrightarrow \Pol(\mathbb{T}^d)$. 

% ??ALTERNATIVE-but I cannot prove the following??
% Further,  $\pi(\Pol(\bK_q))$  contains the element $1 \otimes z_i$ with $z_i: \mathbb{T}^d \rightarrow  \mathbb{T}$ the $i$-th coordinate function. Therefore, the norm closure of  $\pi(\Pol(\bK_q))$ contains the C$^\ast$-algebra $1 \otimes C(\mathbb{T}^d)$. 
 \end{enumerate}
 \end{thm}

We now need to complement the above theorem further, using also the topological characterisation of the spectrum of $\C(\bK_q)$, due to Neshveyev and Tuset (\cite{NeshTuset}).

\begin{prop}\label{Thm=MassiveDecompnew}
Let $\bK_q$ be as above. Then the following hold.	
	\begin{enumerate}	
 \item \label{Item=One-New} Let $\mu$ be the normalized Lebesgue measure on $\mathbb{T}^d$. Consider the representation
\begin{equation}\label{Eqn=Int}
\pi =  \bigoplus_{w \in W}  \int^\oplus_{\mathbb{T}^d}  \pi_w \otimes \pi_\tau \, d\mu(\tau).
\end{equation}
Then $\pi$ extends from $\Pol(\bK_q)$ to a representation of $\C(\bK_q)$ and this representation is moreover faithful. The image $\pi(\C(\bK_q))$ is contained in   $L_\infty(\mathbb{T}^d, \oplus_w B(H_w)) \simeq \oplus_w B(H_w) \otimes L_\infty(\mathbb{T}^d)$. 
\item \label{Item=Two-New} We have  $\pi(\Pol(\bK_q)) \subseteq \oplus_w B( H_w) \odot \Pol(\mathbb{T}^d)$.  Furthermore there exists a Borel set $X \subseteq \mathbb{T}^d$ with non-empty interior and a non-zero element $x \in \overline{\pi(\Pol(\bK_q))}$ such that for every $w \in W$ the space $H_w \otimes L_2(X)$ is in the kernel of $x$. One can choose $X$ to be the $d$-fold Cartesian product of $\bt_{\uparrow}$.
 \end{enumerate}
\end{prop}

\begin{proof}
%	Statements \eqref{Item=One} and \eqref{Item=Two} are \cite[Theorem 6.2.7]{KS}. 
By amenability of $\widehat{\bK_q}$  every representation of $\Pol(\bK_q)$ extends to $\C(\bK_q)$. 
Further, as \eqref{Item=One} in Theorem \ref{Thm=MassiveDecomp} characterizes all irreducible representations of $\Pol(\bK_q)$ we see that $\pi$ must be a faithful representation of $\C(\bK_q)$. It  follows directly from the integral decomposition \eqref{Eqn=Int} that the image of $\Pol(\bK_q)$ is contained in $L_\infty(\mathbb{T}^d; \oplus_w  B(H_w)) \simeq \oplus_w  B(H_w) \otimes L_\infty(\mathbb{T}^d)$. 	
The first statement of \eqref{Item=Two-New} follows from \eqref{Item=Two} of Theorem  \ref{Thm=MassiveDecomp}. For the second statement put $X:=\mathbb{T}_{\uparrow}^d = 	\mathbb{T}_{\uparrow} \times \ldots \times \mathbb{T}_{\uparrow}$, the $d$-fold Cartesian product of $\bt_{\uparrow}$. 
	 %So $\mathbb{T}_{\uparrow}^d \subseteq \mathbb{T}^d$.
Suppose that the second statement does not hold for this set. This would mean that the $\bigcap_{w \in W,t \in \mathbb{T}_{\uparrow}^d} \textup{Ker} \; \pi_{w, t} = \{0\}$, or in other words, the set $\{\pi_{w,t}: w \in W,t \in \mathbb{T}_{\uparrow}^d\}$ would be dense in the spectrum of $\C(\bK_q)$, equipped with Jacobson topology. This however is false, as follows from the special case of results of \cite{NeshTuset}. Indeed, in the notation of that paper, considering the situation where $S= \prod$ and $L$ is trivial  we are reduced to the study of $\Pol(\bK_q)$. Thus Theorem 4.1 (ii) of \cite{NeshTuset} shows in particular that (again, using the notation of that paper)
\begin{align*} \textup{cl}  \{\pi_{w,t}:w \in W,t \in \mathbb{T}_{\uparrow}^d\} \subset \bigcup_{w \in W}  \textup{cl}  \{\pi_{w,t}:t \in \mathbb{T}_{\uparrow}^d\}= 
\bigcup_{w \in W} \{\pi_{\sigma,t}:\sigma \in W, \sigma \leq w, t \in T_{\sigma,w} \mathbb{T}_{\uparrow}^d\}. 
\end{align*}
	 Then if say $w_{max} \in W$ is the longest element of the Weyl group we see that $\pi_{w_{max},t} \in  \textup{cl}  \{\pi_{w,t}:w \in W,t \in \mathbb{T}_{\uparrow}^d\} $ if and only if $t \in \mathbb{T}_{\uparrow}^d$ and the density statement fails.

\end{proof}

We are ready to formulate the main result of this section.

\begin{thm}
	Let  $q\in (0,1)$, let $K$ be a simply connected semisimple compact Lie group and let $\bK_q$ denote the $q$-deformation  of $K$ in the sense of \cite{KS}. Then $\widehat{\bK_q}$ is not $\C^\ast$-unique. 
\end{thm}
\begin{proof}
	 Let $H = \oplus_{w} H_w$.
Let $\mathbb{T}_{\uparrow}^d = 	\mathbb{T}_{\uparrow} \times \ldots \times \mathbb{T}_{\uparrow}$ be the $d$-fold Cartesian product, as in the proof of the last proposition. Suppose that $f$ and $g$ are in  $B(H) \odot \Pol(\mathbb{T}^d)$ and act on $H \otimes L_2(\mathbb{T}^d)$. Let $f_{\uparrow}$ and $g_{\uparrow}$ be their restrictions to the subspace $H \otimes L_2(\mathbb{T}_{\uparrow}^d)$. Then $f = g$ if and only if  $f_{\uparrow} = g_{\uparrow}$. 

Now consider the representation,
\[
\pi_{\uparrow} =  \bigoplus_{w \in W}  \int^\oplus_{\mathbb{T}_{\uparrow}}  \pi_w \otimes \pi_\tau \: d\mu(\tau).
\]
Note that $\pi_{\uparrow}(x) = \pi(x)\vert_{H \otimes L_2(\mathbb{T}_{\uparrow})}$. 
We have from Proposition \ref{Thm=MassiveDecompnew} that $\pi$ maps $\Pol(\bK_q)$ injectively to $B(H) \odot \Pol(\mathbb{T}^d)$ and therefore $\pi_{\uparrow}$ is injective by the first paragraph.  Hence $\Vert x \Vert_{\uparrow} := \Vert \pi_{\uparrow} ( x ) \Vert$ defines a $\C^\ast$-norm on $\Pol(\bK_q)$ which is majorized by the reduced/universal $\C^\ast$-norm. The majorization is proper; indeed, by Theorem \ref{Thm=MassiveDecompnew} \eqref{Item=Two-New} take non-zero $x \in \overline{\pi(\Pol(\bK_q))} \subseteq B(H) \otimes L_\infty(\mathbb{T})$ such that the space   $H \otimes L_2(\mathbb{T}_{\uparrow})$ is in the kernel of $x$. Consider a sequence $x_n$ in $\Pol(\bK_q)$ such that $\pi(x_n)$ converges to $x$ in the norm of $\C(\bK_q)$.  Then $\pi_{\uparrow}(x_n)$ converges to 0 in norm. This concludes that the reduced norm properly majorizes the $\Vert \: \Vert_{\uparrow}$-norm.  		
\end{proof}

  \begin{comment}
 The proof of the corresponding result for the general deformation should go as follows: we denote by $\bG_q$ a general deformation \'a la Korogodsky-Soibelman, and by $\bT$ the maximal torus in $\bG_q$. 
 Let $W$ denote the Weyl group of $G$. Then one can consider a faithful representation of $C^*_t(\bG_q)$ on the space $\bigoplus_{w \in W} H_w \otimes L^2(\bT)$, given essentially just by the `integral' over $\bT$ of the representations of the form $\Pi_{w,t}:= (\pi_w \otimes \tau_t) \circ \Delta$, $w \in W$ and $t \in \bT$, as in Theorem 6.3.1 of \cite{KS}. This follows as the representations $\Pi_{w,t}$ form a complete set of irreducibles and the algebra is type $I$ (the latter by simple inspection of irreducibles). Then I believe that the same trick should work: we restrict the action to the bottom half of $\bT$ and use the fact that the `highest weight' elements have the spectrum of the form $E\mathbb{T}$, where $E$ is a certain discrete set on the plane. It seems however difficult to write the representation as above explicitly -- perhaps we can describe the argument abstractly, it should work. The place to look at is also the article \cite{NeshTuset}.
 
 \end{comment}

\begin{rmk}
The last theorem holds also for $q=1$. More precisely, if $G$ is an infinite compact linear group, then its dual is not $\C^*$-unique. In other words, the unital $^*$-algebra $\Pol(G)$ of coefficients of finite-dimensional unitary representations of $G$ admits non-unique $C^*$-completions. To see it, it suffices to embed concretely $G\subset  U(n)$ by choosing a fundamental representation and choose a non-empty open subset of $G$, say $X$,  whose complement has non-empty interior. Then we can repeat the trick as before, simply using as the relevant Hilbert space $L^2(X)$ (with the Haar measure of $G$), on which $\Pol(G)$ acts by multiplication. We leave the details to the reader. 
\end{rmk}

% ********PLEASE EDIT THE NEXT PART ADAM**************
% Another comment: the same method gives the result for $SU(2)$. In fact for any compact Lie group we can show that $\Pol(G)$ is not $C^*$-unique, essentially by copying Prop. 2.4 of Alexeev-Kyed.   
% The case of $SU(2)$ is of course easiest to write, one just considers the multiplication representation of $C(SU(2))$ on the space $L^2(X)$ with $X \subset SU(2)$ given by 
% $\left\{\begin{bmatrix}\alpha & - \overline{\gamma} \\ \gamma & \overline{\alpha} \end{bmatrix}:
% |\alpha|^2+|\gamma|^2 = 1, \textup{arg} (\gamma) \in (0,\pi)\right\}$.
 
% \begin{thm}
 %	Let $G$ be a non-trivial compact connected semisimple Lie group and let $q \in (0,1]$. Then the discrete quantum group $\hat{G_q}$ is $\C^*$-unique; in other words $\Pol(G_q)$ admits $\C^*$-completions being proper quotients of $C(G_q)$.
% \end{thm}
 
 %\begin{proof}
 %	Consider first the case $q<1$. 

 %	Then consider the classical case $q=1$.
 %\end{proof}

Finally we discuss the corresponding problems for algebras of functions on quantum homogeneous spaces. Consider a compact quantum group $\bG$ and a compact quantum subgroup $\bK$, so that we have a surjective Hopf $^*$-map $q:\Pol(\bG)\to \Pol(\bK)$. We can define then the unital $^*$-algebra $\Pol(\bG/\bK):=\{a\in \Pol(\bG):(\id \otimes q)(\Delta(a)) = a \otimes 1\}$. Note that  $\Pol(\bG/\bK)$ admits a maximal $\C^*$-norm. This in fact extends to any `algebraic core' of an action of a compact quantum group, as is shown for example in Proposition 4.1 and Theorem 4.2 of \cite{Kenny}. This means that Proposition \ref{universal} applies to algebras of such type and naturally we can consider the uniqueness of their $\C^*$-completions. One could then ask whether the theorem on non-$\C^*$-uniqueness of duals of $q$-deformations extends to the appropriate function algebras of quantum homogeneous spaces, as studied for example in \cite{NeshTuset}. The proposition below answers it in the negative, already for the example of the Podle\'s sphere.

 \begin{prop}
 	The algebra of (polynomial) functions on the Podle\'s sphere, i.e.\ the unital $^*$-algebra arising as $\Pol(SU_q(2)/\bt)$, is $\C^*$-unique.
 \end{prop}
 \begin{proof}
 	The result can be proved directly -- using the presentation of $\Pol(SU_q(2)/\bt)$ via generators and relations (see \cite{Podles} or \cite{Dabrowski}). Alternatively, we can use the fact that the universal completion of $\Pol(SU_q(2)/\bt)$ is the unitisation of the algebra of compact operators (hence a just infinite $\C^*$-algebra, as an extension of a simple $\C^*$-algebra by a finite one, see \cite{GrigorchukMusatRordam}) and appeal to
 	Proposition \ref{universal}.
 \end{proof}

 \section{An example of a $\C^*$-unique discrete quantum group which is not locally finite}\label{Sect=Example}
In this section we give an example of a compact quantum group $\bG_q^\theta$ for which $\Pol(\bG_q^\theta)$ has a unique $\C^\ast$-completion. The corresponding discrete quantum group $\widehat{\bG_q^\theta}$   is not locally finite (see Definition \ref{localfin}).   In Question 6.8 of \cite{GrigorchukMusatRordam} the authors ask whether $\C^*$-uniqueness of  a discrete group implies that the group is locally finite.   Our example shows that in the theory of discrete quantum groups the answer to this question is negative.

We construct the example as follows. Throughout this section we fix $q \in (-1,1) \backslash \{ 0 \}$ and an irrational number $\theta \in (0, 1)$.
Let  $\bG_q = SU_q(2)$, with generators $\alpha$ and $\gamma$ in $\Pol(\bG_q)$ satisfying the relations \eqref{Eqn=RelSU2}. Consider the automorphism $\rho_\theta$ of $\Pol(\bG_q)$ given by:
\[
\alpha \mapsto \alpha, \qquad \gamma \mapsto e^{2\pi i \theta} \gamma. 
\]  
Analysing the defining relations of $\Pol(\bG_q)$ we find that $\rho_\theta$ indeed extends to a $\ast$-automorphism $\rho_\theta: \Pol(\bG_q) \rightarrow \Pol(\bG_q)$. Further,  the equalities \eqref{Eqn=SU2Coproduct} imply that $\Delta \circ \rho_\theta = (\rho_\theta \otimes \rho_\theta) \circ \Delta$ so that $\rho_\theta$ is a Hopf-$\ast$-automorphism of $\Pol(\bG_q)$. 
 Let $\bG_q^\theta = \bG_q \rtimes_{\rho_\theta} \mathbb{Z}$ be the resulting  crossed product quantum group (see Section 6 of \cite{FMP} for the details of the construction). More precisely, let $\Pol(\bG_q^\theta)$ be the algebraic crossed product $\Pol(\bG_q) \rtimes_{\rho_\theta} \mathbb{Z}$ which is the  $\ast$-algebra generated by $\Pol(\bG_q)$ and a unitary $u_\theta$ subject to the relations 
\[
u_\theta^\ast  \gamma u_\theta =  e^{2\pi i \theta} \gamma, \qquad 
u_\theta^\ast  \alpha u_\theta = \alpha.
\] 
The coproduct on $\Pol(\bG_q^\theta)$ is the extension of the coproduct from $\Pol(\bG_q)$  obtained by setting
\[
\Delta(u_{\theta}) = u_\theta \otimes u_\theta. 
\]
It follows from the defining relations that indeed $\Delta$ extends to a $\ast$-homomorphism $\Delta: \Pol(\bG_q^\theta) \rightarrow \Pol(\bG_q^\theta) \odot \Pol(\bG_q^\theta)$. %Similarly, the anti-pode $S$ of $\bG_q$ can be extended to a map $S: \Pol(\bG_q^\theta) \rightarrow \Pol(\bG_q^\theta)$ by setting $S(u_\theta) = u_\theta^{-1}$. 
In fact $\Pol(\bG_q^\theta)$ is a Hopf-$\ast$-algebra, and the elements $\alpha^k \gamma^m (\gamma^\ast)^n u_\theta^l$ with $k,m,n \in \mathbb{N}_{0}, l \in \mathbb{Z}$ form a linear basis of $\Pol(\bG_q^\theta)$. %We extend the Haar integral from $\bG_q$ to $\bG_q^\theta$ by, 
%\[
%\varphi(\alpha^k \gamma^m (\gamma^\ast)^n u_\theta^l) = \delta_{l,0} \: \varphi(\alpha^k \gamma^m (\gamma^\ast)^n).
%\]
%So that $\Pol(\bG_q^\theta)$ is a Hopf-$\ast$-algebra with a left and right invariant integral. By general theory \cite{KustermansDaele} its image in the GNS-representation naturally equips $\bG_q^\theta$ with the structure of a compact C$^\ast$-algebriac quantum group $\bG_q^\theta$ such that $\Pol(\bG_q^\theta)$ as defined before is its polynomial algebra.
As explained for example in \cite{FMP} the universal completion of $\Pol(\bG_q^\theta)$ yields the $\C^*$-algebra $\C(\bG_q^\theta)$ fitting into the Woronowicz quantum group theory (and arising simply as the $\C^*$-algebraic crossed product  $\C(\bG_q) \rtimes_{\rho_\theta} \mathbb{Z}$). Basic properties of crossed products by $\bz$ and Theorem \ref{amenableknown} show that $\widehat{\bG_q^\theta}$ is amenable (this will also follow from the results below).

At this point we also recall the definition of the non-commutative torus $\mathbb{T}_\theta$. We define $\Pol(\mathbb{T}_\theta)$ as the $\ast$-algebra generated by two unitaries $v_\theta$ and $w_\theta$  such that $w_\theta v_\theta = e^{2\pi i \theta} v_\theta w_\theta$. Its universal completion is a C$^\ast$-algebra denoted by $\C(\mathbb{T}_\theta)$.   Recall that $\theta \in (0,1)$ is irrational. We record now a well-known fact.

 %Note that so far the condition that $\theta \in [0, 1)$ is irrational has not been used. 

\begin{lem} \label{lem:Polunique}
	The $\ast$-algebra %generated by $u_\theta$ and $\gamma$ is isomorphic to 
	$\Pol(\mathbb{T}_\theta)$ is $\C^\ast$-unique. 
\end{lem}	
\begin{proof}
	This is a consequence of the simplicity of the universal C$^\ast$-algebra $\C(\mathbb{T}_\theta)$. Indeed, $\C(\mathbb{T}_\theta)$ admits a crossed product decomposition $\C(\mathbb{T}) \rtimes \mathbb{Z}$ where $\mathbb{Z}$ acts by means of a rotation by $\theta$. Since this action is topologically transitive as $\theta$ is irrational, we find that $\C(\mathbb{T}_\theta)$ is simple, c.f. \cite{RieffelPacific} for an overview of relevant results. Now suppose that $\pi: \Pol(\mathbb{T}_\theta) \rightarrow B(H)$ is an injective $\ast$-homomorphism. It extends to a $\ast$-homomorphism $\C(\mathbb{T}_\theta) \rightarrow B(H)$ which by simplicity of $\C(\mathbb{T}_\theta)$ is isometric. Therefore, the norm on $\pi(\Pol(\mathbb{T}_\theta))$ equals the norm of   $\C(\mathbb{T}_\theta)$. 
\end{proof}

For a self-adjoint operator $x$ and a set $A \subseteq \mathbb{R}$, we let $p_A(x)$ be its spectral projection. We let $B_\delta(\lambda)$ be the open ball centered at $\lambda \in \mathbb{R}$ with radius $\delta > 0$. 

\begin{lem}\label{Lem=qCom}
Suppose that $x, y \in B(H)$,  $x$ is normal and $xy = q yx$ for some $q \in \mathbb{R} \backslash \{ 0 \}$. If $\lambda \in \sigma(x)$ then  $q \lambda \in \sigma(x)$ or $y p_{B_{\delta}(\lambda)}(x) \rightarrow 0$ as $\delta \rightarrow 0$. 
\end{lem}
\begin{proof}
Suppose that we do not have that $y p_{B_{\delta}(\lambda)}(x) \rightarrow 0$ as $\delta \rightarrow 0$. Then there exists $r>0$, an increasing sequence of positive integers  $(n_k)_{k\in\N}$ and a sequence of unit vectors $\xi_k \in H$ such that for every $k \in \N$ we have  $\Vert y \xi_k \Vert > r$ and $\xi_k \in p_{B_{1/{n_k}}(\lambda)}(x) $. A basic version of the spectral theorem shows that $\|x \xi_k - \lambda \xi_k\|$ tends to $0$ as $k$ tends to infinity. Then $(x - q \lambda) y \xi_k = y q (x -  \lambda)  \xi_k \rightarrow 0$. But this can only happen if $x - q\lambda$ is not invertible. So  $q \lambda \in \sigma(x)$.  
\end{proof}

%We now need some elementary spectral arguments, whose analogues appear for example in Sectiion ?? of \cite{KS}

\begin{lem}\label{Lem=GammaSpec}
	Suppose that $\Pol(  \bG_q^\theta)$ is represented on a Hilbert space $H$. We have $\sigma(\gamma^\ast \gamma) = \{0\} \cup \bigcup_{n \in \mathbb{N}_0} q^{2n}$ and $\sigma(\gamma) = \{0\} \cup \bigcup_{n \in \mathbb{N}_0} q^{n} \mathbb{T}$.
\end{lem}
\begin{proof}	
	Take $\lambda \in \sigma(\gamma^\ast \gamma)$. As $\alpha \gamma^\ast \gamma = q^2 \gamma^\ast \gamma \alpha$ we find by Lemma \ref{Lem=qCom} that $q^{-2} \lambda \in \sigma(\gamma^\ast \gamma)$  or $\alpha p_{B_\delta(\lambda)}(\gamma^\ast \gamma) \rightarrow 0$ as $\delta \searrow 0$. Suppose that $n \in \N_0$ is maximal such that $\lambda q^{-2n} \in \sigma(\gamma)$.  We find that $\alpha p_{B_{\delta}(\lambda q^{-2n})}(\gamma^\ast \gamma) \rightarrow 0$ in norm as $\delta \searrow 0$. Therefore, for $\varepsilon > 0$ we may pick $\delta>0$ small enough so that 
	\[
	\begin{split}
	 \vert 1 - \lambda q^{-2n} \vert = & \Vert    (1 - \lambda q^{-2n}) p_{B_{\delta}(\lambda q^{-2n})}(\gamma^\ast \gamma) \Vert \\
	\leq &  \Vert    (-\alpha^\ast \alpha + 1 - \lambda q^{-2n}) p_{B_{\delta}(\lambda q^{-2n})}(\gamma^\ast \gamma) \Vert + \varepsilon \\
	= &   \Vert    (\gamma^\ast \gamma - \lambda q^{-2n}) p_{B_{\delta}(\lambda q^{-2n})}(\gamma^\ast \gamma) \Vert + \varepsilon  
	\leq 2 \varepsilon.
	\end{split}
	\] 
	Therefore, $\lambda q^{-2n} = 1$. This shows that $\sigma(\gamma^\ast \gamma) \subseteq \cup_{0 \leq n \leq N} q^{2n}$ for some $N \in \mathbb{N}_0 \cup \{ \infty \}$. Applying the same argument to the commutation relation  $q^2 \alpha^\ast \gamma^\ast \gamma =  \gamma^\ast \gamma \alpha^\ast$ shows that in fact $N = \infty$. Indeed, suppose that $N$ is finite. It follows that if $\lambda \in \sigma(\gamma^\ast \gamma)$ then $q^2 \lambda \in \sigma(\gamma^\ast \gamma)$ or  $\alpha^\ast p_{B_\delta(\lambda)}(\gamma^\ast \gamma) \rightarrow 0$ as $\delta \searrow 0$. So we must have $\alpha^\ast p_{B_\delta(q^{2N})}(\gamma^\ast \gamma) \rightarrow 0$ as $\delta \searrow 0$. But this entails that for $\varepsilon > 0$ we may pick $\delta > 0$ small, so that 
	\[
	1 = \Vert (\alpha \alpha^\ast + q^2 \gamma^\ast \gamma) p_{B_\delta(q^{2N})}(\gamma^\ast \gamma) \Vert \leq q^{2 + 2N} +\varepsilon.
	\]
	 This is a contradiction, so $N= \infty$.

	Secondly, since $\gamma$ is normal it generates a commutative C$^\ast$-algebra with spectrum $\sigma(\gamma)$. By the first part and the spectral mapping theorem it follows  that  $\sigma(\gamma) \subseteq \{0\} \cup \bigcup_{n \in \mathbb{N}_0} q^n \mathbb{T}$ and that  for every $n \in \mathbb{N}_0$ we have that $\sigma(\gamma) \cap q^n \mathbb{T}$ is  non-empty. 
	Since $\gamma u_\theta = e^{2\pi i \theta} u_\theta \gamma$ and $u_\theta$ is unitary it follows by Lemma \ref{Lem=qCom} that if $\lambda \in \sigma(\gamma)$ then $e^{2\pi i \theta}  \lambda \in \sigma(\gamma)$. Since $\theta$ is irrational and the spectrum is closed this implies that $\mathbb{T} \lambda \in \sigma(\gamma)$. Hence $\sigma(\gamma) =\{0\} \cup \bigcup_{n \in \mathbb{N}_0} q^{n} \mathbb{T}$. 
\end{proof}

\begin{lem}\label{Lem=UniCon}
	Suppose that $\Pol(\bG_q^\theta)$ is represented on a Hilbert space $H$. For $n \in \mathbb{N}_0$, let $p_n$ be the spectral projection of $\gamma$ corresponding to the circle $q^n \mathbb{T}$ and let $K_n = p_n H$. For every $n \in \mathbb{N}_0$ there exists a unitary  $v_n: K_0 \rightarrow K_n$ such that $v_n^\ast u_\theta  v_n = u_\theta p_0$, $v_n^\ast \gamma  v_n = q^n \gamma p_0$. Further, for each $m,n, k \in \N_0$ we have  $p_m (\alpha^\ast)^k p_n = \delta_{n+k, m} c_q(m,n) v_{m,n}$, where $v_{m,n} = v_m v_n^\ast$ and for $m > n$, 
	\begin{equation}\label{Eqn=Cqn}
	c_q(m,n) = \sqrt{ (1-q^{2m}) (1-q^{2m-2}) \cdots (1-q^{2n+2})}.
	\end{equation}
\end{lem}
\begin{proof}
	Let $\alpha^\ast = u \vert \alpha^\ast \vert$ be the polar decomposition of $\alpha^\ast$ and set $u_n = u p_n$. The operators  $\alpha \alpha^\ast$ and $\gamma \gamma^\ast$ commute and hence generate a commutative C$^\ast$-algebra. $\gamma$ is contractive, c.f. Lemma \ref{Lem=GammaSpec}, so that by the relation $\alpha \alpha^\ast + q^2 \gamma^\ast \gamma = 1$ we see that the support projection $u^\ast u = \supp (\alpha \alpha^\ast) = 1$. By the relation $\alpha^\ast \alpha +  \gamma^\ast \gamma = 1$ we see that   $u u^\ast = \supp(\alpha^\ast \alpha)  = 1 - p_0$.

	Fix now $n \in \N$. We get $u_n^\ast u_n = p_n u^\ast u p_n = p_n$. So that $u_n$ is a partial isometry with range projection $p_n$. Further,  
	\begin{equation}\label{Eqn=GammaCom}
	\begin{split}
	\gamma^\ast \gamma u p_n = \gamma^\ast \gamma \alpha^\ast \vert \alpha^\ast \vert^{-1} p_n = q^2  \alpha^\ast  \gamma^\ast \gamma  \vert \alpha^\ast \vert^{-1} p_n = q^2  \alpha^\ast  \vert \alpha^\ast \vert^{-1}  \gamma^\ast  \gamma    p_n = q^{2+2n} u   p_n,
	\end{split}
	\end{equation}
	so that $u_n$ maps $p_n H$ into $p_{n+1} H$. Similarly, $u^\ast$ maps $p_{n+1} H$ into $p_n H$ and further $u u^\ast p_{n+1} = (1 - p_0) p_{n+1} = p_{n+1}$. This shows that $u_n: p_n H \rightarrow p_{n+1}H$ is unitary. The operators $\alpha$ and $\gamma^\ast \gamma$ commute with $u_\theta$ so that also $u, p_n$ and $u_n$ commute with $u_\theta$. So $u_n^\ast u_\theta u_n = u_\theta p_n$. Further $u_n^\ast \gamma u_n = q \gamma p_n$ by essentially the same computation as \eqref{Eqn=GammaCom}. 
	Setting $v_n = u_{n-1} \cdots  u_0$ then completes the proof of the first statement.

	Now, to show $p_m (\alpha^\ast)^k p_n = \delta_{n+k, m} c_q(m,n) v_{m,n}$ it suffices by induction to show this for $k = 1$.  We already concluded that $\vert \alpha^\ast \vert$ commutes with $p_n$ and $p_m \alpha^\ast p_n = p_m u \vert \alpha^\ast \vert p_n = p_m u  p_n \vert \alpha^\ast \vert = \delta_{n+1,m} \vert \alpha^\ast \vert p_n$. Further $\vert \alpha^\ast \vert^2 + q^2 \gamma \gamma^\ast = 1$, so that $\vert \alpha^\ast \vert^2 p_n + q^{2+2n} p_n = p_n$ and we conclude that $\vert \alpha^\ast \vert p_n = \sqrt{1 - q^{2+2n}} p_n$.
\end{proof}

Since $\gamma^\ast \gamma$ and $u_\theta$ commute we see that $p_n$ and $u_\theta$ commute and that the spaces $K_n$ in Lemma \ref{Lem=UniCon} are invariant subspaces for $u_\theta$. Further $q^{-n} \gamma$ restricted to $K_n$ is unitary as its spectrum equals $\mathbb{T}$. This shows that the restrictions of $u_\theta$ and  $q^{-n} \gamma$ to $K_n$ satisfy the relations of the non-commutative torus $\mathbb{T}_\theta$ and that the prescription $v_\theta \mapsto u_\theta p_n$ and $w_\theta \mapsto q^{-n} \gamma p_n$ gives a non-trivial representation $\pi_n$ of $\Pol(\mathbb{T}_\theta)$.  As $\C(\mathbb{T}_\theta)$ is simple, each $\pi_n$ is faithful (this is essentially equivalent to Lemma \ref{lem:Polunique}).  

\begin{cor}
All the representations in the family $\{\pi_n: \Pol(\mathbb{T}_\theta) \to B(K_n), n \in \mathbb{N}\}$ described above  are unitarily conjugate. 	
\end{cor}
\begin{proof}
	This is a consequence of Lemma \ref{Lem=UniCon}.
\end{proof}

We are ready to state and prove the main result of this section.

\begin{thm}\label{Thm=CastUniqueGq}
The algebra	$\Pol(  \bG_q^\theta)$ is $\C^\ast$-unique.
\end{thm}
\begin{proof}
	Let $\pi$ be a representation of $\Pol(\bG_q^\theta)$  on a Hilbert space $H$. As in Lemma \ref{Lem=UniCon} we decompose $H = \oplus_{n=0}^\infty K_n$. We may moreover assume that all Hilbert spaces $K_n$ are isomorphic, and that the respective unitaries  conjugate the actions of  $\Pol(\mathbb{T}_\theta)   \simeq \Pol \langle  q^{-n}\gamma, u_\theta \rangle p_n$ on $K_n,   n \in  \N_0$.  So we assume that $K = K_n$ and hence $H = \oplus_{n=0}^\infty K$. Moreover, there is a single representation $\pi_{\mathbb{T}_\theta}: \Pol(\mathbb{T}_\theta) \rightarrow B(K)$ such that for  $x \in \Pol(\mathbb{T}_\theta)$ under the identification $H = \oplus_{n=0}^\infty K$ we have $\pi(x) = \oplus_{n=0}^\infty \pi_{\mathbb{T}_\theta}(x)$. For each $n \in \N_0$ let, as in Lemma \ref{Lem=UniCon}, $p_n$ be the spectral projection of $\pi(\gamma^\ast \gamma)$ corresponding to the spectral set $q^n \mathbb{T}$. Then  $p_n$ is by construction the projection onto the $n$-th summand in $H = \oplus_{n=0}^\infty K$ and we set $p_{-1} = 0$. We also find from Lemma \ref{Lem=UniCon} that $\pi(\alpha)p_n = p_{n-1} \pi(  \alpha) = \sqrt{1 - q^{2n}} v p_n$ where $v: \oplus_{n=0}^\infty K \rightarrow \oplus_{n=0}^\infty K$ is the backwards shift  $(\xi_n)_{ n\in \N_0 } \mapsto (\xi_{n+1})_{n\in \N_0}$. In particular, it follows that 
	\[
	\begin{split}
	   p_k \pi(\alpha^m) p_n = & \delta_{n-m,k} c_q(n,k), \qquad m\geq 0, \\
	   p_k \pi( (\alpha^\ast)^m) p_n = & \delta_{n+m,k} c_q(k,n), \qquad m > 0. 
	\end{split}
	\]
	with $c_q(k,n)$ as described in \eqref{Eqn=Cqn} and $c_q(n,n) = 1$. For simplicity write $\alpha^{m}$ for $(\alpha^\ast)^{-m}$ in case $m < 0$.  Also set $c_q(m,n) = c_q(n,m)$ in case $m < n$. Let $Q = \sum_{m = -\infty}^\infty P_m(u_\theta, \gamma, \gamma^\ast) \alpha^m$ with $P_m(u_\theta, \gamma, \gamma^\ast)$  a linear combination of basis vectors $u_\theta^k (\gamma^\ast)^i \gamma^j$ where $k \in \mathbb{Z}, i,j \in \mathbb{N}_0$.  We may identify $B(H)$ with $B(K)  \otimes B(\ell_2)$ and then we see that
	\[
	\pi(Q) =
	\sum_{m,n = 0}^\infty
	\pi_{\mathbb{T}_\theta}( P_m(u_\theta, \gamma, \gamma^\ast)  )  \otimes  c_q(m,n) e_{m,n} \in  B(K)  \otimes B(\ell_2).
	\]
	
	Suppose now that we have two representations $\pi_1$ and $\pi_2$ of $\Pol(\bG_q^\theta)$, on $H^1$ and $H^2$ respectively. Then we get decompositions $H^1 = \oplus_{n=1}^\infty K^1$ and $H^2 = \oplus_{n=1}^\infty K^2$ with representations $\pi^1_{\mathbb{T}_\theta}$ and $\pi^2_{\mathbb{T}_\theta}$ of $\Pol(\mathbb{T}_\theta)$ on $K^1$ and $K^2$ respectively, such that,
	\begin{equation}\label{Eqn=PiImage}
	\pi^i(Q) = 
	\sum_{m,n = 0}^\infty
	\pi_{\mathbb{T}_\theta}^i( P_m(u_\theta, \gamma, \gamma^\ast)  )  \otimes  c_q(m,n) e_{m,n} \in  B(K^i)  \otimes B(\ell_2).
	\end{equation}
	Since $\C(\mathbb{T}_\theta)$ is simple we find that 
	\[
	\pi^1_{\mathbb{T}_\theta}(\C(\mathbb{T}_\theta)) \simeq_{ \pi^1_{\mathbb{T}_\theta} }\C(\mathbb{T}_\theta) \simeq_{\pi^2_{\mathbb{T}_\theta}} \pi^2_{\mathbb{T}_\theta}(\C(\mathbb{T}_\theta))
	\]
	 are isomorphic. From \eqref{Eqn=PiImage} we see that the complete isometry  $( \pi^1_{\mathbb{T}_\theta} \circ (\pi^2_{\mathbb{T}_\theta})^{-1} \otimes \id_{\ell_2})$ maps $\pi^1(\Pol(\bG_q^\theta))$ bijectively to  $\pi^2(\Pol(\bG_q^\theta))$. Therefore the norms on $\pi^i(\Pol(\bG_q^\theta))$ with $i=1,2$ are equal and we conclude that $\Pol(\bG_q^\theta)$ is $\C^\ast$-unique. 	
\end{proof}

\begin{thm}
	The discrete quantum group $\widehat{\bG_q^\theta}$ is $\C^*$-unique and not locally finite.
\end{thm}
\begin{proof}
C$^\ast$-uniqueness is proved in Theorem \ref{Thm=CastUniqueGq}. This quantum group is not locally finite as the quantum subgroup generated by $u_\theta$ is not finite. In fact one may check that any non-empty choice of finitely many generators of the fusion ring of $\widehat{\bG_q^\theta}$ generates an infinite quantum group.   	
\end{proof}

\subsection*{Acknowledgements} Key work on this paper was done during the visit of the second author to TU Delft in November 2018; the hospitality of the mathematics department is gratefully acknowledged.
The second author was also partially supported by the National Science Center (NCN) grant no.~2014/14/E/ST1/00525.  We thank the referee for careful reading of our manuscript and thoughtful comments.


\begin{thebibliography}{999999}





\bibitem[AK]{AlexeevKyed} V.\,Alexeev and D.\,Kyed, Uniqueness questions for $\C^*$-norms on group rings, \emph{Pacific J. Math.} \textbf{298} (2019), no.\,2, 257--266.

\bibitem[Ban]{Banica} T.\,Banica, Representations of compact quantum groups and subfactors, \emph{J.\,Reine Angew.\,Math.} \textbf{509} (1999), 167--198.

\bibitem[BMT]{BMT}
E.\,B\'edos, G.\,Murphy and L.\,Tuset,
Co-amenability for compact quantum groups,
\emph{J.\,Geom.\,Phys.} \textbf{40} (2001) no.\,2, 130--153.

\bibitem [BO]{bo} N.\,Brown and N.\,Ozawa, ``$\C^*$-Algebras and finite dimensional approximations'', Graduate Studies in Mathematics,
88. American Mathematical Society, Providence, RI, 2008.


\bibitem[Dab]{Dabrowski} L.\,Dabrowski, The garden of quantum spheres, \emph{in} ``Noncommutative geometry and quantum groups (Warsaw, 2001)'',  \emph{Banach Center Publ.} \textbf{61} (2003), 37–-48.

\bibitem[Daw]{Daws}
M.\,Daws, Remarks on the quantum Bohr compactification, \emph{Illinois J.\,Math.} \textbf{57} (2013), no.\,4, 1131--1171.



\bibitem[DC]{Kenny}K.\,De Commer, Actions of compact quantum groups, \emph{Banach Center Publ.} \textbf{111} (2017), 33--100.

\bibitem[DKSS]{DKSS} K.\,De Commer, P.\,Kasprzak, A.\,Skalski and P.M.So\l tan, Quantum actions on discrete quantum spaces and a generalization of Clifford's theory of representations, \emph{Israel J. Math.} \textbf{226} (2018), no.\,1, 475--503.

\bibitem[DQV]{Vaesetal}  P.\,Desmedt, J.\,Quaegebeur and S.\,Vaes, Amenability and the bicrossed product construction, \emph{Illinois J.\,Math.} \textbf{46} (2002), no.\,4, 1259--1277.

\bibitem[FMP]{FMP} P.\,Fima, K.\,Mukherjee and I.\,Patri, On compact bicrossed products, \emph{J.\,Noncommut.\,Geom.} \textbf{11} (2017), no.\,4, 1521--1591.


\bibitem[FST]{FST} U.\,Franz, A.\,Skalski and R.\,Tomatsu, Idempotent states on compact quantum groups and their classification on $U_q(2)$, $SU_q(2)$, and $SO_q$(3), \emph{J. Noncommut. Geom.} \textbf{7} (2013), no.\,1, 221--254.

\bibitem[GMR]{GrigorchukMusatRordam} R.\,Grigorchuk, M.\,Musat and M.\,R\o rdam, Just-infinite $\C^*$-algebras, \emph{Comment. Math. Helv.} \textbf{93} (2018), no.\,1, 157–-201.

\bibitem[KS]{KS}
L.I.\,Korogodski and Ya.S.\,Soibelman,   ``Algebras of functions on quantum groups. Part I.'', Mathematical Surveys and Monographs \textbf{56}, Providence, RI: Amer. Math. Soc., 1998.

\bibitem[Kus]{Kustermans} J.\,Kustermans, Locally compact quantum groups in the universal setting, \emph{Internat.\,J.\,Math.} \textbf{12} (2001),  289--338.

%\bibitem[KuDa]{KustermansDaele}
%  J. Kustermans and A.\,Van Daele,
%  \emph{C$^\ast$-algebraic quantum groups arising from algebraic quantum groups},
%   Internat. J. Math. {\bf 8} (1997), no. 8, 1067--1139. 

\bibitem[NT]{NeshTuset} S.\,Neshveyev and L.\,Tuset, Quantized algebras of functions on homogeneous spaces with Poisson stabilizers, \emph{Comm. Math. Phys.} \textbf{312} (2012), no. 1, 223–-250. 



\bibitem[Pod]{Podles} P.\,Podle\'s, Quantum spheres, \emph{Lett. Math. Phys.} \textbf{14} (1987), no.\,3, 193–-202. 


\bibitem[Rie]{RieffelPacific}
  M.\,Rieffel, 
  \emph{$\C^\ast$-algebras associated with irrational rotations},
  Pacific J. Math. {\bf 93} (1981), no. 2, 415--429.


\bibitem[Sol]{Soltan} P.M.\,So\l tan, Quantum Bohr compactification, \emph{Illinois J.\,Math.} \textbf{49} (2005), no. 4, 1245--1270.


\bibitem[Tak]{TakII}
   M.\,Takesaki, 
   Theory of operator algebras. II.
   Encyclopaedia of Mathematical Sciences, 125. Operator Algebras and Non-commutative Geometry, 6. Springer-Verlag, Berlin, 2003. 

\bibitem[Tom]{Tomatsu} R.\,Tomatsu, Amenable discrete quantum groups, \emph{J.\,Math.\,Soc.\,Japan} \textbf{58} (2006), no.\,4, 949--964.



\begin{comment}
  \bibitem[BKS97]{BKS}
M. Bozejko, B. K\"ummerer, R. Speicher, 
\emph{$q$-Gaussian processes: non-commutative and classical aspects},
Comm. Math. Phys. {\bf 185} (1997), no. 1, 129--154.


\bibitem[BrOz08]{NateTaka}
  N. Brown, N.  Ozawa, 
  \emph{C$^\ast$-algebras and finite-dimensional approximations},
   Graduate Studies in Mathematics, 88. American Mathematical Society, Providence, RI, 2008. xvi+509 pp. 


\bibitem[Cas18]{Caspers}
M. Caspers, 
\emph{Gradient forms and strong solidity of free quantum groups},
preprint.  


\bibitem[CaSk15]{CaspersSkalskiCMP}
   M. Caspers, A. Skalski, 
   \emph{The Haagerup approximation property for von Neumann algebras via quantum Markov semigroups and Dirichlet forms},
    Comm. Math. Phys. {\bf 336} (2015), no. 3, 1637--1664. 

\bibitem[Cip97]{Cipriani}
  F. Cipriani,
  \emph{Dirichlet forms and Markovian semigroups on standard forms of von Neumann algebras},
   J. Funct. Anal. {\bf 147}  (1997), 259–-300.

  \bibitem[CiSa03]{CiprianiSauvageot}
F. Cipriani, J.-L. Sauvageot, 
\emph{Derivations as square roots of Dirichlet forms},
J. Funct. Anal. {\bf 201} (2003), no. 1, 78--120.


  \bibitem[CiSa17]{CiprianiSauvageotAdvances}
F. Cipriani, J.-L. Sauvageot, 
\emph{Amenability and subexponential spectral growth rate of Dirichlet forms on von Neumann algebras},
 Adv. Math. {\bf 322} (2017), 308--340.


\bibitem[DaLi92]{DaviesLindsay}
 B. Davies, M. Lindsay, 
 \emph{Noncommutative symmetric Markov semigroups},
  Math. Z. {\bf 210} (1992), no. 3, 379--411.

\bibitem[GoLi95]{GoldsteinLindsay}
  S. Goldstein, M. Lindsay, 
  \emph{KMS-symmetric Markov semigroups}, 
  Math. Z. 219(4), 591--608 (1995).
  
  
\bibitem[Jol04]{JolissaintMartin}
   P. Jolissaint, F. Martin,  
  \emph{Alg\`ebres de von Neumann finies ayant la propri\'et\'e de Haagerup et semi-groupes L$^2$-compacts},  
    Bull. Belg. Math. Soc. Simon Stevin {\bf 11} (2004), no. 1, 35--48. 
   
    
    \bibitem[NiDy93]{NicaDykema}
    K. Dykema, A. Nica, 
    \emph{On the Fock representation of the $q$-commutation relations},
     J. Reine Angew. Math. {\bf 440} (1993), 201--212. 
    

\bibitem[OzPo10]{OzawaPopaAJM}
 N. Ozawa, S. Popa, 
 \emph{On a class of II$_1$ factors with at most one Cartan subalgebra, II.}
  Amer. J. Math. {\bf 132} (2010), no. 3, 841--866.  

\bibitem[Shl96]{Shlyakhtenko}
  D. Shlyakhtenko, 
  \emph{Random Gaussian band matrices and freeness with amalgamation},
   Internat. Math. Res. Notices 1996, no. {\bf 20}, 1013--1025.

    \bibitem[Was17]{Wasilewski}
    M. Wasilewski,
    \emph{$q$-Araki-Woods algebras: extension of second quantisation and Haagerup approximation property},
     Proc. Amer. Math. Soc. {\bf 145} (2017), no. 12, 5287--5298. 
   
 \end{comment}  
 
\end{thebibliography}
\end{document}